\documentclass{birkjour}

\usepackage{amsmath, amsfonts, amssymb, latexsym}
\usepackage{tabularx, longtable, makecell, arydshln, float}
\usepackage{multirow, multicol, array, placeins}
\usepackage[mathscr]{eucal}
\usepackage[normalem]{ulem}
\usepackage{graphicx, xcolor}
\usepackage{algpseudocode, algorithm, comment}
\usepackage[numbers]{natbib}
\allowdisplaybreaks
\usepackage{enumitem}
\usepackage[justification=centering]{caption}

\newcommand{\HH}{\mathcal{H}}

\newcommand{\R}{\mathbb{R}}

\newtheorem{defi}{Definition}

\newtheorem{thm}{Theorem}
\newtheorem{prop}{Proposition}

\newtheorem{lem}{Lemma}
\newtheorem{rem}{Remark}

\def\beq{\begin{equation}}
\def\eeq{\end{equation}}
\def\beqs{\begin{equation*}}
\def\eeqs{\end{equation*}}

\begin{document}

\title[A splitting algorithm for fixed points]{A splitting algorithm for fixed points of nonexpansive mappings and equilibrium problems}

\author{Le Dung Muu}
\address{Institute of Mathematics and Applied Sciences\\
Thang Long University\\
Hanoi\\
Vietnam}
\email{ldmuu@math.ac.vn}

\author{Xuan Thanh Le}
\address{Institute of Mathematics\\
Vietnam Academy of Science and Technology\\
Hanoi\\
Vietnam}
\email{lxthanh@math.ac.vn}

\begin{abstract}
We consider the problem of finding a fixed point of a nonexpansive mapping,
which is also a solution of a pseudo-monotone equilibrium problem,
where the bifunction in the equilibrium problem is the sum of two ones.
We propose a splitting algorithm combining the gradient method for equilibrium
problem and the Mann iteration scheme for fixed points of nonexpansive mappings.
At each iteration of the algorithm, two strongly convex subprograms are required
to solve separately, one for each of the component bifunctions.
Our main result states that, under paramonotonicity property of the given bifunction,
the algorithm converges to a solution without any Lipschitz type
condition as well as H\"older continuity of the bifunctions involved.
\end{abstract}



\maketitle


\section{Introduction}

Let $\HH$ be a real Hilbert space endowed with weak topology defined by the inner product
$\langle \cdot , \cdot \rangle$ and its induced norm $\| \cdot \|$. Let $C \subseteq \HH$ be a
nonempty closed convex subset and $f: \HH \times \HH \to \R \cup \{+\infty\}$ a bifunction
such that $f(x, y) < +\infty$ for every $x, y \in C$. The equilibrium problem defined by the
Nikaido-Isoda-Fan inequality that we are going to deal with in this paper is given as
\beqs
\text{Find} \ x \in C \text{ such that } f(x, y) \geq 0 \text{ for all } y \in C. \eqno(EP)
\eeqs
In 1955, Nikaido and Isoda \cite{NI1955} first used this inequality in convex game models.
Then in 1972 Ky Fan \cite{F1972} called this inequality a minimax one and established existence
theorems for $(EP)$. After the appearance of the paper by Blum and Oettli \cite{BO1994},
this problem has been attracted much attention of researchers.
In \cite{BCPP2013, BO1994, MO1992} it has been shown that some important problems such as
optimization, variational inequality, Kakutani fixed point, and Nash equilibria can be formulated
in the form of $(EP)$. Many papers concerning the solution existence, stabilities as well as
algorithms for $(EP)$ have been published (see e.g.
\cite{DMQ2016, HM2011, IS2003, M2003, MQ2009, MQ2015, QAM2012, QMH2008, SS2011}
and the excellent survey paper \cite{BCPP2013}).

Recently the problem of finding a solution of an equilibrium problem
which is also a fixed point of a nonexpansive mapping
has been considered in some papers
(see e.g. \cite{S2012, TT2007a, TT2007b, VSN2015} and the references therein).
The existing methods combine algorithms for solving $(EP)$ such as the projection, extragradient,
and proximal point methods with iterative schemes for finding fixed points of nonexpansive mappings.
These methods require either computing the projection onto the feasible domain $C$, or
solving convex and/or strongly monotone regularized equilibrium subproblems
(see e.g. \cite{A2013, AM2014, HMA2016, LY2009, TT2007a, TT2007b, VSN2015}).
However, in general, solving these subproblems is computational cost.
In order to reduce the computational cost,
several splitting algorithms have been developed for some classes of
maximal monotone operator inclusion, variational inequality,
and equilibrium problems (see e.g. \cite{AH2017, CM2014, DM2016, ES2009,
HV2017, HM2017, M2009, P1979, T2000}).

In this paper we  propose splitting algorithms for finding a point
in the intersection of the fixed point set of a finite number of
nonexpansive mappings and the solution set of an equilibrium problem,
where the bifunction is the sum of two bifucntions.
The algorithm is a combination between the gradient method for
equilibrium problem and the Mann iteration scheme for fixed point of nonexpansive mappings.
The main features of the proposed algorithm are the followings:

 $\bullet$ At each iteration, it requires solving two strongly convex programs,
one for each component bifunction separably rather than for their sum;

 $\bullet$ Evaluating each nonexpansive mapping can be done in parallel;

 $\bullet$ Convergence of the proposed algorithms is ensured without
any Lipschitz type or H\"older conditions that are required in some existing
splitting algorithms for equilibrium problems (e.g. in \cite{AH2017, HV2017}).

The remaining part of the paper is organized as follows. The next section are preliminaries
containing some lemmas that will be used in proving the convergence of  our proposed algorithms.
Section \ref{SectionAlgorithm} is devoted to the formulation of our considered problem,
the description of the proposed algorithm, and its convergence analysis.
Section \ref{SectionApplications} shows some variants of the algorithm
when applying to solve some special cases of the problem.
The last section closes the paper with some conclusions.


\section{Preliminaries}

We recall the following well-known definition on monotonicity of bifunctions (see e.g. \cite{BCPP2013}).
\begin{defi}
A bifunction $f: \HH \times \HH \to \R \cup \{+\infty\}$ is said to be
\begin{itemize}[leftmargin = 0.5 in]
\item[(i)] strongly monotone on $C$ with modulus $\beta > 0$ (shortly $\beta$-strongly monotone) if
  $$f(x, y) + f(y, x) \leq -\beta \| y - x \|^2  \quad \forall x, y \in C;$$

\item[(ii)] monotone on $C$ if
  $$f(x, y) + f(y, x) \leq 0 \quad \forall x, y \in C;$$

\item[(iii)] strongly pseudo-monotone on $C$ with modulus $\beta > 0$
(shortly $\beta$-strongly pseudo-monotone) if
  $$f(x, y) \geq 0 \implies f(y, x) \leq -\beta\| y - x \|^2 \quad \forall x, y \in C;$$

\item[(iv)] pseudo-monotone on $C$ if
  $$f(x, y) \geq 0 \implies f(y, x) \leq 0 \quad \forall x, y \in C;$$

\item[(v)] paramonotone on $C$ with respect to a set $S$ if
 $$x^* \in S, x\in C \text{ and } f(x^*, x) = f(x, x^*) = 0 \text{ implies } x \in S.$$
\end{itemize}
\end{defi}

Obviously, $(i) \implies (ii) \implies (iv)$ and $(i) \implies (iii) \implies (iv)$.
Note that a strongly pseudo-monotone bifunction may not be monotone.
Paramonotone bifunctions have been used in e.g. \cite{AM2014, SS2011, S2011}.
Clearly in the case of optimization problem when
$f(x,y) = \varphi(y) - \varphi(x)$, the bifunction  $f$ is paramonotone on $C$ with
respect to the solution set of the problem $\min_{x\in C} \varphi(x)$.
Conditions for a bifunction $f$ to be paramonotone can be found in \cite{I2003}.

The following well known lemmas will be used for proving the convergence of the algorithm
proposed in the next section.

\begin{lem}\label{lem1}{\em (see \cite{TX1993} Lemma 1)}
Let $\{\alpha_k\}$ and $\{\sigma_k\}$ be two sequences of nonnegative numbers such that
$\alpha_{k+1} \leq \alpha_k + \sigma_k$ for all $k \in \mathbb{N}$,
where $\sum_{k=1}^{\infty} \sigma_k < \infty$. Then the sequence $\{\alpha_k\}$ is convergent.
\end{lem}

\begin{lem}\label{lemyen}
Let $\HH$ be a real Hilbert space with the inner product $\langle \cdot , \cdot \rangle$
and its induced norm $\| \cdot \|$. Then for $x, y, z \in \HH$ and $0\leq \gamma \leq 1$, one has
\beqs
\| \gamma x + (1-\gamma)y - z \|^2
= \gamma \|x-z\|^2 + (1-\gamma) \|y-z\|^2 - \gamma(1-\gamma) \|x-y\|^2.
\eeqs
\end{lem}
\begin{proof} By definition of the inner product and its reduced norm we have
\begin{align*}
& \ \| \gamma x + (1-\gamma) y - z \|^2\\
= & \ \| \gamma (x-z) + (1-\gamma) (y-z) \|^2\\
= & \ \gamma^2 \| x - z \|^2 + (1 - \gamma)^2 \| y - z \|^2
+ 2 \gamma (1 - \gamma) \langle x - z, y - z \rangle\\
= & \ \gamma \| x - z \|^2 + (1 - \gamma) \| y - z \|^2\\
& - \gamma (1 - \gamma) \left( \| x - z \|^2 + \| y - z \|^2 - 2 \langle x - z, y - z \rangle \right)\\
= & \ \gamma \| x - z \|^2 + (1 - \gamma) \| y - z \|^2 - \gamma(1 - \gamma) \|(x-z) - (y-z)\|^2\\
= & \ \gamma \| x - z \|^2 + (1 - \gamma) \| y - z \|^2 - \gamma(1 - \gamma) \|x-y\|^2.
\end{align*}
This proves the lemma.
\end{proof}


\section{Problem formulation, algorithm and its convergence}\label{SectionAlgorithm}

\subsection{The problem and its special cases}\label{ProblemStatement}

Let $T: C \to C$ be a nonexpansive mapping, that is
\beqs
\|T(x)-T(y)\| \leq \|x-y\| \ \forall x, y \in C.
\eeqs
The set of all fixed points of the mapping $T$ is denoted by $Fix(T)$.
Let $f: \HH \times \HH \to \mathbb{R}$ be a bifunction.
In what follows we suppose that $f(x,y) = f_1(x,y) + f_2(x,y)$
and that $f_i(x,x) = 0$ ($i=1,2$) for every $x, y \in C$.
The following assumptions will be used in the sequel.
\begin{itemize}[leftmargin = 0.5 in]
\item[(A1)] For each $x \in C$, the functions $f_1(x, \cdot)$ and
$f_2(x, \cdot)$ are convex, subdifferentiable on an open set containing $C$,
while the function $f(\cdot, x)$ is weakly upper semicontinuous on $C$.
\item[(A2)] The bifunction $f$ is pseudo-monotone on $C$.
\item[(A3)] Either $\text{int} C \neq \emptyset$ or,
for every $x\in C$, each function  $f_i(x, \cdot)$ is continuous at a point of $C$.
\end{itemize}
The main problem we are considering in this paper is to find a fixed point of $T$
which is also an equilibrium point of $f$ on $C$. More formally, the problem is stated
as follows.
\beqs
\text{Find} \ x^* \in C \text{ such that } x^*= T(x^*) \text{ and } f(x^*,y) \geq 0 \text{ for all } y \in C. \eqno(P)
\eeqs

Let us mention some typical examples for Problem $(P)$.\\

{\it 1. Equilibrium problem over the set of common fixed points of nonexpansive mappings.}
Let $T_i: C \to C (i = 1, \ldots, m)$ be nonexpansive mappings. Consider the following problem
\begin{align*}
&\text{Find} \ x^* \in C \text{ such that }\\
&x^*= T_i(x^*) \text{ for all $i = 1, \ldots, m$ and } f(x^*,y) \geq 0 \text{ for all } y \in C.
\end{align*}
This problem can be casted into Problem $(P)$, thanks to the following lemma.
\begin{lem} {\rm(see \cite{BC2010} Proposition 4.34).}\label{CommonFixedPoints}
Let $\mu_i > 0$ $(i = 1, \ldots, m)$, $\sum_{i=1}^m  \mu_i = 1$ and $T(x) := \sum_{i=1}^m \mu_1 T_i(x)$
for every $x \in C$. Then $T$ is nonexpansive on $C$ and its fixed point set coincides the intersection
of the fixed point sets of $T_i$ ($i = 1, \ldots ,m$).
\end{lem}

{\it 2. Equilibrium problem over the intersection of closed convex sets.} Consider the problem
\beqs
\text{Find} \ x \in C := \bigcap_{i=1}^m C_i \text{ such that } f(x,y) \geq 0 \ \text{ for all } y \in C, \eqno(P_1)
\eeqs
where $C_i (i = 1, \ldots, m)$ are closed convex sets. In this case, we can take
$T_i(x) := P_{C_i}(x)$ for each $i = 1, \ldots, m$ (i.e., $T_i$ is the projection map on $C_i$),
and take $T(x) := \sum_{i = 1}^m \mu_i T_i(x)$ with  $0 < \mu_i < 1$ for every $i$,
 $\sum_{i=1}^m \mu_i = 1$. Then by Lemma \ref{CommonFixedPoints} we have $Fix(T) \equiv C$,
 and therefore Problem $(P_1)$ can be formulated in form of Problem $(P)$.\\

{\it 3. Common solution of equilibrium problem and maximal monotone operator inclusion}. Consider the problem
\begin{align}
&\text{Find } x \in C \text{ such that } \notag\\
&f(x,y) \geq 0 \text{ for all } y \in C \text{ and } 0 \in M_i(x) \text{ for all }  i = 1, \ldots, m, \tag{$P_2$}
\end{align}
where $M_i (i=1, \ldots, m)$  is maximal monotone multi-valued operators on $\HH$.
It is well-known (see e.g. \cite{R1976}) that the operator $T_i:= (M_i+ cI)^{-1}$ with $c >0$ is
defined everywhere, single-valued, nonexpansive on the whole space and
its fixed point set coincides with the solution set of the inclusion $0 \in M_i(x)$.
Thus Problem $(P_2)$ can be reformulated as $(P)$.\\

{\it 4. Split equilibrium problem}. The split feasibility problem
introduced in \cite{CE1994} is given as
\beqs
\text{Find} \ x \in U \text{ such that } Ax \in V, \eqno(SFP)
\eeqs
where $U$ and $V$ are respectively nonempty closed convex subsets of Hilbert spaces $\HH_1$ and $\HH_2$,
and $A$ is a bounded linear operator from $\HH_1$ to $\HH_2$. Let us consider this problem with
$U$ being the solution set of equilibrium problem $(EP)$ and the inclusion $Ax \in V$ being represented by
the solution set of the system of inequalities
\beqs
\langle a^i, x \rangle \leq b_i \quad (i = 1, \ldots, m).
\eeqs
In this setting, Problem $(SFP)$ can be written as
\begin{align}
&\text{Find} \ x \in C \text{ such that }\notag\\
&f(x,y) \geq 0 \text{ for all } y \in C \text{ and } \langle a^i, x \rangle \leq b_i \ (i = 1, \ldots, m). \tag{$SEP$}
\end{align}
Let $H_i$ be the half space $\{x \in \HH \ | \ \langle a^i, x \rangle - b_i \leq 0$.
Then Problem $(SEP)$ can take the form of Problem $(P)$ with
$T(x):= \sum_{i=1}^m \mu_i P_{H_i}(x)$, where $P_{H_i}$ is the projection operator onto the half space $H_i$,
and $\mu_i (i = 1, \ldots, m)$ are positive real numbers such that $\sum_{i=1}^m \mu_i = 1$.


\subsection{The algorithm and its convergence analysis}

The algorithm below is a combination between the grandient one for pseudo-monotone
equilibrium problem $(EP)$ and the Mann iterative scheme for finding fixed points of
the nonexpansive mapping $T$. The stepsize is computed as in the algorithm for
equilibrium problem in \cite{SS2011}.

\begin{algorithm}[H]
\caption{A splitting algorithm for solving $(P)$.} \label{AlgorithmSplit}
\begin{algorithmic}
\State \textbf{Initialization:} Seek $x^0\in C$.
Choose $\gamma \in (0,1)$ and a sequence $\{\beta_k\}_{k \geq 0} \subset \mathbb{R}$
satisfying the following conditions
\beqs
\quad \sum_{k=0}^\infty \beta_k = +\infty, \quad \sum_{k=0}^\infty \beta_k^2 < +\infty.
\eeqs
\State \textbf{Iteration} $k = 0, 1, \ldots$:
\State \qquad Take $g_1^k \in \partial_2 f_1(x^k, x^k), g_2^k \in \partial_2 f_2(x^k, x^k)$.
\State \qquad Compute
\begin{align}\label{Mann}
\eta_k &:= \max\{\beta_k, \|g_1^k\|, \|g_2^k\|\}, \ \lambda_k := \dfrac{\beta_k}{\eta_k},\notag\\
y^k &:= \arg\min\{\lambda_k f_1(x^k, y) + \dfrac{1}{2}\|y - x^k\|^2 \mid y \in C\},\notag\\
z^k &:= \arg\min\{\lambda_k f_2(x^k, y) +\dfrac{1}{2}\|y - y^k\|^2 \mid y \in C\},\notag\\
x^{k+1} &:= \gamma z^k + (1-\gamma)T(x^k).
\end{align}
\end{algorithmic}
\end{algorithm}

In order to prove the convergence of Algorithm \ref{AlgorithmSplit}, we need the auxiliary
results in the following propostions. For that we denote by $\Omega$ the solution set of
Problem $(P)$ and assume that $\Omega \neq \emptyset$.

\begin{prop}\label{DistanceConvergence}
 For each $x^* \in \Omega$, the sequence $\{\|x^k-x^*\|\}_{k \in \mathbb{N}}$ is convergent.
\end{prop}
\begin{proof}
To simplify the notations, for each $k \geq 0$ let
\begin{align*}
h_1^k (x) &:= \lambda_k f_1(x^k,x) + \frac{1}{2}\| x-x^k \|^2,\\
h_2^k (x) &:= \lambda_k f_2(x^k,x) + \frac{1}{2}\| x-y^k \|^2.
\end{align*}
By Assumption (A1), the function $h_1^k$ is strongly convex with modulus $1$ and
sub-differentiable, which implies
\beq \label{ct1}
h_1^k(y^k) + \langle u_1^k,x-y^k \rangle + \frac{1}{2}\| x-y^k \|^2 \leq h_1^k (x) \quad \forall x \in C
\eeq
for any $u_1^k \in \partial h_1^k (y^k)$. As defined in Algorithm
\ref{AlgorithmSplit}, $y^k$ is a minimizer of $h_1^k(\cdot)$ over $C$.
Therefore, by Assumption (A3) and the optimality condition for convex programming, we have
\beqs
0 \in \partial h_1^k (y^k) + N_C (y^k),
\eeqs
which implies that there exists $u_1^k \in - \partial h_1^k (y^k)$ such that
$ \langle u_1^k,x-y^k \rangle \geq 0$ for all $x \in C$. Hence, for each $x \in C$,
it follows from (\ref{ct1}) that
$$ h_1^k(y^k) + \frac{1}{2} \| x - y^k \|^2 \leq h_1^k(x),$$
i.e.,
$$\lambda_k f_1(x^k, y^k) + \frac{1}{2}\| y^k - x^k \|^2 + \dfrac{1}{2}\|x - y^k\|^2
\leq \lambda_k f_1(x^k, x) + \frac{1}{2}\| x - x^k \|^2,$$
or equivalently,
\begin{equation} \label{ct2}
\|y^k - x\|^2 \leq \|x^k - x\|^2 +2\lambda_k \left( f_1(x^k, x)-f_1(x^k, y^k) \right) - \|y^k - x^k\|^2.
\end{equation}
By the same argument on $h_2^k(\cdot)$ and $z^k$, we have
\begin{equation} \label{ct3}
\|z^k-x\|^2 \leq \|y^k-x\|^2 +2\lambda_k\left(f_2(x^k,x)-f_2(x^k,z^k)\right) - \|z^k-y^k\|^2.
\end{equation}
Combining (\ref{ct2}) and (\ref{ct3}) yields
\begin{align}\label{ct4}
\|z^k-x\|^2 \leq \|x^k-x\|^2 & + 2\lambda_k f(x^k,x) - \|y^k-x^k\|^2 - \|z^k-y^k\|^2\notag\\
& - 2\lambda_k\left(f_1(x^k,y^k) + f_2(x^k,z^k)\right).
\end{align}
Since $g_1^k \in \partial_2f_1(x^k,x^k)$ and $f_1(x^k,x^k) = 0$, we have
\beqs
f_1(x^k,y^k) = f_1(x^k,y^k) - f_1(x^k,x^k) \geq \langle g_1^k, y^k - x^k \rangle,
\eeqs
which implies
\beq\label{ct5}
-2 \lambda_k f_1(x^k,y^k) \leq - 2 \lambda_k \langle g_1^k, y^k - x^k \rangle.
\eeq
By Cauchy-Schwarz inequality and the fact that $\|g_1^k\| \leq \eta_k$, from (\ref{ct5}) we have
\beq \label{ct6}
-2 \lambda_k f_1(x^k,y^k) \leq 2 \dfrac{\beta_k}{\eta_k}\eta_k\|y^k - x^k\| = 2\beta_k\|y^k-x^k\|.
\eeq
By the same argument, we obtain
\beq \label{ct7}
-2 \lambda_k f_2(x^k,z^k) \leq 2 \beta_k\|z^k-x^k\|.
\eeq
Replacing (\ref{ct6}) and (\ref{ct7}) to (\ref{ct4}) we get
\begin{align} \label{ct8}
\|z^k-x\|^2
&\leq \|x^k-x\|^2 + 2 \lambda_k f(x^k,x)\notag\\
& - \|y^k-x^k\|^2 - \|z^k - y^k\|^2 + 2\beta_k \left(\|y^k-x^k\| + \|z^k-x^k\|\right)\notag\\
&= \|x^k-x\|^2 + 2 \lambda_k f(x^k,x)\notag\\
& + 2\beta_k^2 - \left(\|y^k-x^k\| - \beta_k\right)^2 - \left(\|z^k-x^k\| - \beta_k\right)^2\notag\\
&\leq \|x^k-x\|^2 + 2 \lambda_k f(x^k,x) + 2\beta_k^2.
\end{align}
Taking $x = x^* \in \Omega \subseteq C$ in (\ref{ct8}) we get
\beq\label{ct9}
\|z^k-x^*\|^2 \leq \|x^k-x^*\|^2 + 2 \lambda_k f(x^k,x^*) + 2\beta_k^2.
\eeq
Furthermore, since $x^{k+1} = \gamma z^k + (1-\gamma) T(x^k)$ as defined
in Algorithm \ref{AlgorithmSplit}, we have
\begin{align}\label{ct10}
\|x^{k+1} - x^*\|^2 &= \|\gamma z^k + (1-\gamma) T(x^k) - x^*\|^2 \notag\\
&= \gamma \|z^k-x^*\|^2 + (1-\gamma)\|T(x^k)-T(x^*)\|^2\notag\\
&\ \ \ - \gamma(1-\gamma)\|z^k - T(x^k)\|^2\notag\\
&\leq \gamma\|z^k-x^*\|^2 + (1-\gamma)\|x^k-x^*\|^2 - \gamma(1-\gamma)\|z^k - T(x^k)\|^2\notag\\
&\leq \gamma\left(\|x^k-x^*\|^2 + 2 \lambda_k f(x^k,x^*) + 2\beta_k^2\right) + (1-\gamma)\|x^k-x^*\|^2\notag\\
&\ \ \ - \gamma(1-\gamma)\|z^k - T(x^k)\|^2\notag\\
&= \|x^k-x^*\|^2 + 2 \gamma \lambda_k f(x^k,x^*) + 2\gamma\beta_k^2\notag\\
&\ \ \ - \gamma(1-\gamma)\|z^k - T(x^k)\|^2.
\end{align}
Here, the second equality follows from Lemma \ref{lemyen} and the fact that $T(x^*) = x^*$,
the first inequality is due to the non-expansiveness of the mapping $T$,
the second inequality is a consequence of (\ref{ct9}),
while the last equality is trivial.
Now we note that $f(x^*, x^k) \geq 0$ since $x^*$ belongs to the solution set of $(P)$.
This implies that $f(x^k, x^*) \leq 0$ by pseudo-monotonicity of the bifunction $f$ on $C$.
From (\ref{ct10}), by the negativity of $f(x^k, x^*)$ and due to $\gamma \in (0, 1)$, we obtain
\beq\label{ct11}
\|x^{k+1} - x^*\|^2 \leq \|x^k-x^*\|^2 + 2\gamma\beta_k^2.
\eeq
Since $\gamma > 0$ and $\sum_{k = 1}^{\infty} \beta_k^2 < \infty$,
in virtue of Lemma \ref{lem1}, the inequality (\ref{ct11}) implies
that the sequence $\{\|x^k - x^*\|\}_{k \in \mathbb{N}}$ is convergent.
This closes the proof of the proposition.
\end{proof}

\begin{prop}\label{WeakCluster}
 Any weakly cluster point of $\{x^k\}_{k \in \mathbb{N}}$ is a fixed point of $T$.
\end{prop}
\begin{proof}
In the following we will show that $\|T(x^k) - x^k\| \to 0$ as $k \to \infty$.
The proposition follows immediately from this claim.

Indeed, by taking $x = x^k$ in (\ref{ct8}) and note that $f(x^k, x^k) = 0$, we obtain
\beqs
\|z^k - x^k\|^2 \le 2 \beta_k^2,
\eeqs
which implies
\beq\label{ct12}
\lim_{k \to \infty} \|z^k - x^k\| = 0,
\eeq
since $\beta_k \to 0$ as $k \to \infty$. On the other hand, let $x^* \in \Omega$ be fixed, then
  $f(x^k,x^*) \leq 0$.
Therefore, from (\ref{ct10}) we have
\begin{align*}
\gamma(1-\gamma)\|T(x^k)-z^k\|^2
&\leq \|x^k-x^*\|^2 - \|x^{k+1}-x^*\|^2 + 2 \gamma \lambda_k f(x^k,x^*) + 2\gamma\beta_k^2\\
&\leq \|x^k-x^*\|^2 - \|x^{k+1}-x^*\|^2 + 2\gamma\beta_k^2,
\end{align*}
which implies
\beq\label{ct13}
\lim_{k \to \infty}\|T(x^k)-z^k\| = 0,
\eeq
since $\{\|x^k-x^*\|\}$ is convergent, $\gamma \in (0, 1)$, and $\beta_k \to 0$ as $k \to \infty$.
To the end, by (\ref{ct12}) and (\ref{ct13}), we obtain
\beqs
\|T(x^k)-x^k\| \leq \|T(x^k) - z^k\| + \|z^k - x^k\| \to 0 \ \text{as} \ k \to \infty.
\eeqs
This closes the proof of the proposition.
$\square$
\end{proof}

We now establish  the convergence  result in the following theorem.

\begin{thm} \label{thm1}
 Suppose that $f$ is paramonotone on $C$ with respect to the solution set
 $Sol(C,f)$ of problem $(EP)$. Then under the assumptions \emph{(A1)}, \emph{(A2)}, \emph{(A3)},
 the sequence $\{x^k\}_{k \in \mathbb{N}}$ generated by Algorithm \ref{AlgorithmSplit}
 converges weakly to a solution of $(P)$, provided that $(P)$ admits a solution.
\end{thm}
\begin{proof}
Let $x^*$ be in the solution set $\Omega$ of $(P)$. As obtained in the
proof of Proposition \ref{DistanceConvergence}, from (\ref{ct10})
and the negativity of $f(x^k, x^*)$ we have
\begin{align*}
0 &\leq - 2 \gamma \lambda_k f(x^k, x^*)\\
&\leq \|x^k - x^*\|^2 - \|x^{k+1} - x^*\|^2 + 2 \gamma\beta_k^2 - \gamma(1-\gamma)\|z^k - T(x^k)\|^2\\
&\leq \|x^k - x^*\|^2 - \|x^{k+1} - x^*\|^2 + 2 \gamma \beta_k^2
\end{align*}
for every $k \in \mathbb{N}$, which implies  that
\beq\label{ct14}
0 \leq - 2 \gamma \displaystyle \sum_{k = 0}^\infty \lambda_k f(x^k, x^*)
\leq \|x^0 - x^*\|^2 + 2 \displaystyle \sum_{k = 0}^\infty \beta_k^2 < +\infty,
\eeq
since $\sum_{k = 0}^{\infty} \beta_k^2 < +\infty$. On the other hand,
note that the sequences $\{g_1^k\}_{k \in \mathbb{N}}$ and $\{g_2^k\}_{k \in \mathbb{N}}$
are bounded by Proposition 4.1 \cite{VSN2015}. This fact, together with the construction
of $\{\beta_k\}_{k \in \mathbb{N}}$, implies that
there exists $M > 0$ such that $\|g_1^k\| \leq M, \|g_2^k\| \leq M, \beta_k \leq M$
for all $k \in \mathbb{N}$. Hence, for each $k \in \mathbb{N}$ we have
\beqs
\eta_k = \max\left\{\beta_k, \|g_1^k\|, \|g_2^k\|\right\} \leq M,
\eeqs
which implies
\beqs
\lambda_k = \dfrac{\beta_k}{\eta_k} \geq \dfrac{\beta_k}{M}.
\eeqs
Since $\sum_{k = 0}^{\infty} \beta_k = +\infty$, it follows that
\beq\label{ct15}
\displaystyle \sum_{k=0}^{\infty} \lambda_k = +\infty.
\eeq
The combination of (\ref{ct14}) and (\ref{ct15}) implies that
\beqs
\limsup \{f(x^k, x^*) \ | \ k \in \mathbb{N}\} = 0.
\eeqs
Let $\{x^{k_j}\}_{j \in \mathbb{N}}$ be a subsequence of $\{x^{k}\}_{k \in \mathbb{N}}$ such that
\beqs
\lim_{j \to +\infty} f(x^{k_j},x^*) = \limsup f(x^k,x^*) = 0.
\eeqs
By Proposition \ref{DistanceConvergence}, the sequence $\{\|x^k-x^*\|\}_{k \in \mathbb{N}}$
is convergent. It follows that the sequence $\{x^k\}_{k \in \mathbb{N}}$ is bounded,
and hence its subsequence $\{x^{k_j}\}_{j \in \mathbb{N}}$ is also bounded.
We may therefore assume that $\{x^{k_j}\}_{j \in \mathbb{N}}$ weakly converges
to some $\bar{x} \in C$. Since $f(\cdot,x^*)$ is weakly upper semicontinuous, we have
\beq\label{ct16}
f(\bar{x}, x^*) \geq \lim_{j \to +\infty} f(x^{k_j},x^*) = 0,
\eeq
and as a consequence, $f(x^*, \bar{x}) \leq 0$ by pseudo-monotonicity of the bifunction $f$.
On the other hand, $f(x^*, \bar{x}) \geq 0$ since $x^*$ belongs to the solution set $\Omega$ of $(P)$.
Therefore we obtain
\beq\label{ct17}
f(x^*, \bar{x}) = 0.
\eeq
This implies $f(\bar{x}, x^*) \leq 0$ by pseudo-monotonicity of $f$. Together with (\ref{ct16}),
it follows that
\beq\label{ct18}
f(\bar{x}, x^*) = 0.
\eeq
Since $\bar{x} \in C, x^* \in \Omega \subset Sol(C, f)$, and $f$ is paramonotone on $C$
with respect to $Sol(C, f)$, from (\ref{ct17}) and (\ref{ct18}) we have $\bar{x} \in Sol(C, f)$.
Furthermore, since $\bar{x}$ is a weakly cluster point of $\{x^k\}_{k \in \mathbb{N}}$,
by Proposition \ref{WeakCluster}, $\bar{x}$ is a fixed point of $T$.
Hence $\bar{x} \in Sol(C, f) \cap Fix(T) = \Omega$. It therefore follows from
Proposition \ref{DistanceConvergence} that the sequence
$\left\{\|x^k - \bar{x}\|\right\}_{k \in \mathbb{N}}$ converges.
Note that $\left\{x^{k_j}\right\}_{j \in \mathbb{N}}$ weakly converges to $\bar{x}$,
we can conclude that the whole sequence $\left\{x^k\right\}_{k \in \mathbb{N}}$ weakly
converges to $\bar{x}$, which is a solution to $(P)$.
\end{proof}

\begin{rem}
When $\mathcal{H}$ is a finite dimensional space, Assumption (A3) can be omitted (see e.g. \cite{Tuy2016} page 70).
\end{rem}


\section{Applications}\label{SectionApplications}

In this section, we apply Algorithm \ref{AlgorithmSplit} to some special cases of
Problem $(P)$ mentioned in Section \ref{ProblemStatement}.
For equilibrium problem over the intersection of closed convex sets $(P_1)$,
by taking $T_i(x) := P_{C_i}(x)$ for each $i = 1, \ldots, m$, the computation of
$x^{k+1}$ in (\ref{Mann}) takes the form
\beqs
x^{k+1} := \gamma z^k + (1-\gamma) \sum_{i=1}^m \mu_i P_{C_i}(x^k).
\eeqs
So in this case, we obtain a splitting algorithm for Problem $(P_1)$,
where optimization problems are solved separately for each function $f_1$ and $f_2$,
while the projection is computed in parallel onto each convex set $C_i$
rather than onto their intersection.

Similarly, applying Algorithm \ref{AlgorithmSplit} to find a
common solution of equilibrium problem and maximal monotone operator inclusion
(Problem $(P_2)$), the iterate $x^{k+1}$ is computed separately for each resovent operator
by taking
\beqs
x^{k+1} := \gamma z^k + (1-\gamma) \sum_{i=1}^m \mu_i (M_i + cI)^{-1} (x^k).
\eeqs

To illustrate the proposed algorithm for split equilibrium problem $(SEP)$,
let us consider a game with $n$-players.
Each player $i = 1, \ldots, n$ can take an individual action,
which is represented by $x_i \in \mathbb{R}$.
All players together can take a collective action $x = (x_1, \ldots, x_n) \in \mathbb{R}^n$.
Each player $i$ uses a payoff function $f_i$ which depends on actions of other players.
The Nikaido-Isoda function of the game is defined as
\beqs\label{NikaidoIsoda}
f(x,y) :=  \sum_{i=1}^n \left(f_i(x)- f_i(x[y_i])\right),
\eeqs
where the vector $x[y_i]$ is obtained from $x$ by replacing component $x_i$ by $y_i$.
Let $C_i \subset \mathbb{R}$ be the strategy set of player $i$,
then the strategy set of the game is $C := C_1 \times ...\times C_n$.
As usual,  a point $x^* \in C$ is said to be a Nash equilibrium point of the game if
\beqs
f_i(x^*) = \max _{y_i \in C_i} f_i(x^*[y_i]) \quad \forall i = 1, \ldots, n.
\eeqs
It is well known that $x^*$ is an equilibrium point if and only if
$f(x^*,y) \geq 0$ for all $y \in C$.
A concrete practical equilibrium model, where the bifunction is a paramonotone
one being the sum of two monotone functions can be found in \cite{QMH2008}.
In some practical games such as jointly constrained Nash-Cournot equilibrium models,
the equilibrium points are required to satisfy additional constraints given by
\beqs
\langle a^j, x \rangle \leq b_j \quad (j = 1, \ldots, m).
\eeqs
Such the game has exactly the form of Problem $(SEP)$.
For this problem, the computation of iterate point $x^{k+1}$ in (\ref{Mann})
of Algorithm \ref{AlgorithmSplit} takes the form
\beqs
x^{k+1} = \mu z^k + (1-\mu)\sum_{j=1}^m P_{H_j}(x^k),
\eeqs
where $H_j$ is the half space defined by the inequality $\langle a^j, x \rangle - b_j \leq 0$,
and therefore the projection $P_{H_j}$ has a closed form.

\section{Conclusion}

We have proposed a splitting algorithm for finding a point in the intersection of
solution set of a pseudo-monotone equilibrium problem and the fixed point set
of a nonexpansive mapping. The bifunction involved in the equilibrium problem
is the sum of the two ones. Exploiting this special structure, the proposed
splitting algorithm requires solving two strongly convex subprograms separately for each component bifunction.
  Combining with the Mann iteration scheme, the algorithm converges
under the paramonotonicity property of the involved bifunction.
Some variants of the algorithm devoted to some special cases of the considered problem
have been shown.

\section*{Acknowledgements}

This work is supported by National Foundation for Science and Technology Development
(NAFOSTED) of Vietnam under grant number 101.01-2017.315.


\begin{thebibliography}{1}

\bibitem{AH2017}
Anh, P.K., Hai, T.N.:
\newblock Splitting extragradient-like algorithms for strongly pseudomonotone equilibrium problems.
\newblock Numer. Algor. {\bf 76}(1), 67--91 (2017)

\bibitem{A2013}
Anh, P.N.:
\newblock A hybrid extragradient method extended to fixed point problems and equilibrium problems.
\newblock Optimization. {\bf 62}(2), 271--283 (2013)

\bibitem{AM2014}
Anh, P.N., Muu, L.D.:
\newblock A hybrid subgradient algorithm for nonexpansive mappings and equilibrium problems.
\newblock Optim. Lett. {\bf 8}(2), 727--738 (2014)

\bibitem{BC2010}
Bauschke, H.H., Combettes, P.H.:
\newblock Convex Analysis and Monotone Operator in Hilbert Spaces.
\newblock Springer (2010)

\bibitem{BCPP2013}
Bigi, G., Castellani, M., Pappalardo, M., Passacantando, M.:
\newblock Existence and solution methods for equilibria.
\newblock Eur. J. Oper. Res. {\bf 227}(1), 1--11 (2013)

\bibitem{BO1994}
Blum, E., Oettli, W.:
\newblock From optimization and variational inequalities to equilibrium problems.
\newblock Math. Student. {\bf 63}(1-4), 123--145 (1994)

\bibitem{CE1994}
Censor, Y., Elfving, T.:
\newblock A multiprojection algorithm using Bregman projections in a product space.
\newblock Numer. Algor. {\bf 8}(2), 221--239 (1994)

\bibitem{CM2014}
Cruz, J.Y.B., Mill\'an, R.D.:
\newblock A direct splitting method for nonsmooth variational inequalities.
\newblock J. Optim. Theory  Appl. {\bf 161}(3), 728--737 (2014)

\bibitem{DM2016}
Duc, P.M., Muu, L.D.:
\newblock A splitting algorithm for a class of bilevel equilibrium problems involving nonexpansive mappings.
\newblock Optimization. {\bf 65}(10), 1855--1866 (2016)

\bibitem{DMQ2016}
Duc, P.M., Muu, L.D., Quy, N.V.:
\newblock Solution-existence and algorithms with their convergence rate for
strongly pseudomonotone equilibrium problems.
\newblock Pac. J. Optim. {\bf 12}(4), 833--845 (2016)

\bibitem{ES2009}
Eckstein, J., Svaiter, A.F.:
\newblock General projective splitting methods for sums of maximal monotone operators.
\newblock SIAM J. Control Optim. {\bf 48}(2), 787--811 (2009)

\bibitem{F1972}
Fan, K.:
\newblock A minimax inequality and applications.
\newblock In: Shisha, O. (eds) Inequality, vol. III, pp. 103–113. Academic Press, New York (1972)

\bibitem{HV2017}
Hai, T.N., Vinh, N.T.:
\newblock Two new splitting algorithms for equilibrium problems.
\newblock  Rev. R. Acad. Cienc. Exactas, F\'is. Nat. Serie A. Matem{\'a}ticas. {\bf 111}(4), 1051--1069 (2017)

\bibitem{HM2017}
Hieu, D.V., Moudafi, A.:
\newblock A barycentric projected-subgradient algorithm for equilibrium problems.
\newblock J. Nonlinear Var. Anal. {\bf 1}(1), 43--59 (2017)

\bibitem{HMA2016}
Hieu, D.V., Muu, L.D., Anh, P.K.:
\newblock Parallel hybrid extragradient methods for pseudomonotone equilibrium
problems and nonexpansive mappings.
\newblock Numer. Algor. {\bf 73}(1), 197--217 (2016)

\bibitem{HM2011}
Hung, P.G., Muu, L.D.:
\newblock The Tikhonov regularization extended to equilibrium problems involving pseudomonotone bifunctions.
\newblock Nonlinear Anal. {\bf 74}(17), 6121--6129 (2011)

\bibitem{I2003}
Iusem, A.N.:
\newblock On some properties of paramonotone operators.
\newblock J. Convex Anal. {\bf 5}, 269--278 (1998)

\bibitem{IS2003}
Iusem, A.N., Sosa, W.:
\newblock Iterative algorithms for equilibrium problems.
\newblock Optimization. {\bf 52}(3), 301--316 (2003)

\bibitem{LY2009}
Liduka, H., Yamada, I.:
\newblock A subgradient algorithm for the equilibrium problem over the fixed point set and its application.
\newblock Optimization. {\bf 58}(2), 251--261 (2009)


\bibitem{M2003}
Mastroeni, G.:
\newblock Gap functions for equilibrium problems.
\newblock J. Global Optim. {\bf 27}(4), 411--426 (2003)

\bibitem{M2009}
Moudafi, A.:
\newblock On the convergence of splitting proximal methods for equilibrium problems in Hilbert spaces.
\newblock J. Math. Anal. Appl. {\bf 359}(2), 508--513 (2009)

\bibitem{MO1992}
Muu, L.D., Oettli, W.:
\newblock Convergence of an adaptive penalty scheme for finding constrained equilibria.
\newblock Nonlinear Anal. {\bf 18}(12), 1159--1166 (1992)

\bibitem{MQ2009}
Muu, L.D., Quoc, T.D.:
\newblock Regularization algorithms for solving monotone Ky Fan
inequalities with application to a Nash-Cournot equilibrium model.
\newblock J. Optim. Theory  Appl. {\bf 142}(1), 185--204 (2009)

\bibitem{MQ2015}
Muu, L.D., Quy, N.V.:
\newblock On existence and solution methods for strongly pseudomonotone equilibrium problems.
\newblock Vietnam J. Math. {\bf 43}, 229--238 (2015)

\bibitem{NI1955}
Nikaid{\^o}, H., Isoda, K.:
\newblock Note on noncooperative convex games.
\newblock Pac. J. Math. {\bf 5}(5), 807--815 (1955)

\bibitem{P1979}
Passty, G. B.:
\newblock Ergodic convergence to a zero of the sum of monotone operators in Hilbert space.
\newblock J. Math. Anal. Appl. {\bf 72}(2), 383--390 (1979)

\bibitem{QAM2012}
Quoc, T.D., Anh, P.N., Muu, L.D.:
\newblock Dual extragradient algorithms extended to equilibrium problems.
\newblock J. Global Optim. {\bf 52}(1), 139--159 (2012)

\bibitem{QMH2008}
Quoc, T.D., Muu, L.D., Hien, N.V.:
\newblock Extragradient algorithms extended to equilibrium problems.
\newblock Optimization. {\bf 57}(6), 749--776 (2008)

\bibitem{R1976}
Rockafellar, R.T.:
\newblock Monotone operators and the proximal point algorithm.
\newblock SIAM J. Control Optim. {\bf 5}, 877--890 (1976)

\bibitem{SS2011}
Santos, P.S.M., Scheimberg, S.:
\newblock An inexact subgradient algorithm for equilibrium problems.
\newblock Comput. Appl. Math. {\bf 30}(1), 91--107 (2011)

\bibitem{S2012}
Sun, S.:
\newblock An alternative regularization method for equilibrium problems
and fixed point of nonexpansive mappings.
\newblock J. Appl. Math. 2012, Article ID 202860 (2012). doi: 10.1155/2012/202860

\bibitem{S2011}
Svaiter, B.F.:
\newblock On weak convergence of the Douglas-Rachford method.
\newblock SIAM J. Control Optim. {\bf 49}(1), 280--287 (2011)

\bibitem{TT2007a}
Tada, A., Takahashi, W.:
\newblock Weak and strong convergence theorems for a nonexpansive mapping and an equilibrium problem.
\newblock J. Optim. Theory Appl. {\bf 133}(3), 359--370 (2007)

\bibitem{TT2007b}
Takahashi, S., Takahashi, W.:
\newblock Viscosity approximation methods for equilibrium problems
and fixed point problems in Hilbert spaces.
\newblock J. Math. Anal. Appl. {\bf 331}(1), 506--515 (2007)

\bibitem{TX1993}
Tan, K.-K., Xu, H.-K.:
\newblock Approximating fixed points of nonexpansive mappings by the Ishikawa iteration process.
\newblock J. Math. Anal. Appl. {\bf 178}, 301--308 (1993)

\bibitem{T2000}
Tseng, P.:
\newblock A modiﬁed forward-backward splitting method for maximal monotone mappings.
\newblock SIAM J. Control Optim. {\bf 38}(2), 431--446 (2000)

\bibitem{Tuy2016}
Tuy, H.: Convex Analysis and Global Optimization, Second Edition. Springer (2016)

\bibitem{VSN2015}
Vuong, P.T., Strodiot, J.-J., Nguyen, V.H.:
\newblock On extragradient-viscosity methods for solving equilibrium
and fixed point problems in a Hilbert space.
\newblock Optimization. {\bf 64}(2), 429--451 (2015)

\end{thebibliography}
\end{document}